\documentclass{article}
\usepackage{verbatim}
\usepackage{amssymb,amsmath,amsthm,bm,bbm, hyperref}
\usepackage{graphicx,xcolor}
\usepackage{epsfig}
\newtheorem{theorem}{Theorem}
\newtheorem{corollary}{Corollary}[section]
\newtheorem{lemma}[corollary]{Lemma}
\newtheorem{proposition}[corollary]{Proposition}

\newcommand{\Prob} {{\mathbb P}}
 
\newcommand{\E}{{\mathbb  E}}

\newcommand{\R}{{\mathbb{R}}}
\newcommand{\C}{{\mathbb C}}

\def \Im {{\rm Im}}

\def \p {\partial}

\def \Half {{\mathbb H}}

\def \diam {{\rm diam}}

\def \eset {\emptyset}

\newenvironment{definition}[1][Definition]{\begin{trivlist}
\item[\hskip \labelsep {\bfseries #1}]}{\end{trivlist}}

\def \F {{\mathcal F}}
\newcommand{{\pe}}  {\partial_e}
\newcommand {{\lodd}} {{\mathcal J}}
\newcommand{{\inrad}} {{\rm inrad}}

\newcommand {{\cent}} {{\bf c}}
\newcommand {{\eb}}  {{\bf e}}
\newcommand{{\osc}} {{\rm osc}}
\newcommand{{\dyadic}}  {{\mathcal Q}}

\newcommand{{\partition}} {{\mathcal P}}
\newcommand {{\bloop}}  {{\mathcal O}}
\newcommand {{\crad}}  {{\rm crad}}
\newcommand {{\wind}} {{\rm wind}}
\newcommand {{\z}} {{\bf z}}
\newcommand {{\whoknows}} {{\mathcal A}}
\newcommand{{\exc}}{{\mathcal E}}
\newcommand{{\loops}} {{\loop}}
\newcommand {{\N}}  {{\mathbb N}}
\newcommand {{\squares}}{{\mathbb A}}

\newcommand {{\lattice}}  {{\mathcal L}}
\def \bgamma {{\bm \gamma}}

\def \hcap {{\rm hcap}}

\def \eset {\emptyset}

\def \bgamma  {{\boldsymbol{\gamma}}}

\author{Gregory F. Lawler  \thanks{Research supported
by National Science Foundation grant DMS-1513036} \qquad Stephen Yearwood\\
Department of Mathematics\\
University of Chicago}

\title{A new proof of reversibility of $SLE_{\kappa}$ for $\kappa \leq 4$}
\begin{document}
   \maketitle
\begin{abstract}

    We give a new proof of the reversibility of the Schramm Loewner evolution for $\kappa \leq 4$. The main ideas used in the proof are similar to those used in the original proof of this result, given by Zhan.
\end{abstract}
   \section{Introduction}
   \label{intro}
The Schramm Loewner evolution ($SLE_{\kappa}$) is a one parameter family of probability measures on curves in the plane parametrized by
$\kappa = 2/a > 0$.  It   gives the only  candidate for scaling limits of  discrete lattice models exhibiting
conformal invariance in the continuum limit.   We recall the
definition; see \cite{lawler-book} for more detail.
  Let $\Half$ denote the upper half plane in $\C$.  If
 $\gamma:[0,\infty) \rightarrow \overline \Half$ is a curve, we write
 $\gamma_t$ for $\gamma[0,t]$ and $D_t$ for the unbounded component
 of $\Half \setminus \gamma_t$. 
 Chordal $SLE_{\kappa}$ from $0$ to $\infty$ is defined (up to a choice of
parametrization) as a random  curve $\gamma \colon [0, \infty) \to \Half$
with $\gamma(0) = 0$ such that the following holds. Let $g_t$ be
the ``mapping-out function'', that is, the unique conformal transformation
$g_t: D_t \rightarrow \Half$ with $g_t(z) = z + o(1)$ as $z \rightarrow
\infty$.
  Then $g_t$ satisfies 
\begin{equation} \label{loewner}
\dot{g}_{t}(z) = \frac{a}{g_{t}(z) - U_t} \qquad g_0(z) = z
\end{equation}
Under this parameterization, $\hcap(\gamma_t) = at$, where $\hcap$
denotes the half-plane capacity.  
 In other words, we have the
expansion
\[   g_t(z) = z + \frac{at}{z} + O(|z|^{-2}), \;\;\;\;
  z \rightarrow \infty.\]

If $z \in \mathbb{C} \setminus 0$, the solution to \eqref{loewner} holds  for
all $ t < T_z$ where $$T_z = \underset{t}{\text{sup}} \lbrace \text{min} \lbrace |g_s(z) - U_s| : 0 \leq s \leq t \rbrace > 0 \rbrace .$$
To get chordal $SLE_\kappa$ connecting distinct boundary
points in a simply connected domain $D$, one takes the conformal image of this under a conformal map.  Again this is defined up to a change of parametrizaton.

While $SLE$ is a model for curves in equilibrium, the definition uses conditional
probabilities given the path up to a certain time and hence adds an artificial
dynamic. One disadvantage is that some properties that are expected of the
limit curve, in particular reversibility, do not follow immediately. 
 Zhan showed this to be true \cite{zhan-reversibility} for $\kappa \leq 4$, while
 Miller and Sheffield were able to extend these results to $\kappa \in (0,8)$ \cite{ig1, ig2, ig3} by realizing $SLE_{\kappa}$ curves as flow lines of the Gaussian free field.
 
The purpose of this note is to give a new proof of reversibility for $\kappa \leq 4$;  we hope in future work to extend this to $4 < \kappa < 8$ to give a proof that
does not make use of the tools of the Gaussian free field.  While we say that it is a new proof, the basic idea of the proof is the same as that given by Zhan.  
Our hope is that we our argument simplifies some of the
details.   We write $SLE$ for $SLE_\kappa$. 

\begin{itemize}

\item  We compare $SLE $ from $0$ to $x$ in $\Half$ to $SLE $ from $x$ to $0$.
These are probability measures on bounded curves $\gamma$ with $\hcap(\gamma) < \infty$.  While $\hcap(\gamma)$
is a random quantity, it is almost immediate from the definition that the
distribution of $\hcap(\gamma)$  is the same for $SLE$ in both directions.

\item We view $SLE$ connecting two points in $\R$
 as a probability measure on the final mapping-out
functions $g_\gamma$.

\item  We then focus on $SLE$ from $0$ to $x$ and $x$ to $0$ conditioned
to have a specific half-plane capacity.  We show that these two probability
measures agree on the conformal maps $g_\gamma$
for each value of $\hcap[\gamma]$.  By scaling it suffices
to prove this for all $x$ assuming $\hcap[\gamma] = a$.  

\item For each $r \in [0,1]$ we consider the probability measure
$\mu_r$ which corresponds to the following:
\begin{itemize} 
\item  Take $SLE$ from $0$ to $x$ conditioned to have $\hcap = a$
stopped at time $r$, that is, when $\hcap = ra$ giving $\gamma^1$.
\item Given $\gamma^1$, let $\gamma^2$ be $SLE$ from $x$ to $\gamma^1(r)$
in $\Half \setminus \gamma^1$ conditioned so that $\hcap(\gamma^1
\cup \gamma^2) = a$. 
\item  Output $g_\gamma$ where
 $\gamma = \gamma^1 \oplus \tilde \gamma $
where $\tilde \gamma$ is the reversal of $\gamma^2$.

\end{itemize}
This gives a probability measure on transformations $g_\gamma$  
with $\hcap[\gamma] = a$ which we denote by $\mu_r$.

\item  We consider this as a measure on continuous
functions on a fixed closed ball $K = K_h
\subset \Half$ where $h$ is large enough so that $\Im[g_\gamma(z)]
 \geq a$ for all $z \in K$ and $\hcap[\gamma] = a$.
 We show that the Prokhorov distance between
$\mu_r$ and $\mu_s$ is less than $c|s-r|^{1+\delta}$ for 
some $\delta > 0$.  We conclude that $\mu_s$ is a constant
function of $s$.  In particular, $\mu_0 = \mu_1$ which is the
main result. 

\item  The main local commutation relation which 
is similar to the relations in  \cite{zhan-reversibility}
and  \cite{dubedat-duality} 
   is expressed in terms of Radon-Nikodym
derivatives of independent $SLE$ paths tilted by a Brownian loop term.  This relation is nicest for $\kappa \leq 4$, but we discuss the $\kappa < 8$ case
here in order to prepare for future work.

\end{itemize}

The paper is organized as follows. In Section \ref{prelim}, we review
 $SLE$ connecting two points on the boundary, together with some other basic notation, and then we state the main theorem of this paper. In Section \ref{comm}, we describe the commutation relation, and show explicitly that the measures under consideration have the same Radon-Nikodym derivative with respect to a particular measure. In Section \ref{main} we prove the main theorem in a sequence of steps, relying on a few Loewner chain estimates. Finally, in Section \ref{finalsec}, we give the (delayed) proof of a basic Bessel process fact.

Throughout this paper we fix $\kappa = 2/a \in (0,8)$ and allow 
constants, both implicit and explicit,  to depend
on $\kappa$.  We write just $SLE$ for $SLE_\kappa$.
  For a number of the results, we need $\kappa \leq 4$ and we
say that.  Let
\[  b = \frac{6-\kappa}{2\kappa} = \frac{3a-1}{2} \]
be the boundary scaling exponent.

\section {\texorpdfstring{$SLE$}{Lg} in \texorpdfstring{$\Half$}{Lg} from \texorpdfstring{$x_1$}{Lg} to \texorpdfstring{$x_2$}{Lg}} \label{prelim}

There are   several equivalent characterizations of $SLE$ connecting two real points;  here we will take the perspective of   $SLE$ from $0$ to $x \in \R\setminus \{0\}$ as $SLE$ from $0$ to $\infty$ in $\Half$ tilted by the partition function $\Psi$. For simply connected domains and locally analytic boundary
points $z,w$, the partition 
function for $SLE$ is $\Psi_D(z,w) = H_{\p D}(z,w)^b$.  Here
$H_{\p D}(z,w)$ is the boundary Poisson kernel normalized so that $H_\Half(0,x)
 = x^{-2}$. 
We also define $\Psi_\Half(x,\infty) = 1$ for all $x \in \Half$. 
The partition function
satisfies the scaling rule: if $f: D \rightarrow f(D)$ is a conformal transformation,
then 
\[             \Psi_D(z,w) = |f'(z)|^b \, |f'(w)|^b \, \Psi_{f(D)}(f(z), f(w)). \]
Although this definition of $\Psi_D(z,w)$ requires that $z,w$ be locally analytic boundary points,   ratios of partition functions
can often be defined using the scaling rule as we will see below.

Suppose that $g_t$ satisfies \eqref{loewner} where $U_t = - B_t$ is a standard
Brownian motion defined on a probability space $(\Omega,\F,\Prob)$.  Let $\gamma(t)$
denote the corresponding $SLE_\kappa$ curve and we write $\gamma_t = \gamma(0,t]$.  Under
the measure $\Prob$, $\gamma$ has the distribution of an $SLE_\kappa$ path from $0$
to $\infty$.  We will tilt the measure $\Prob$ using an appropriate local martingale
to get $SLE$ from $0$ to $x$.   Suppose $x \in \R \setminus \{0\}$, and let
 $X_t  = g_t(x) - U_t$ and   $T = T_x= \inf \{ t>0: X_t =0 \}$.   For $t<T$, let $D_t$ be the unbounded component of $\Half \setminus \gamma_t$, and define the local martingale $M_t$
 formally by  
\[ M_t  = x^{1-3a}\dfrac{\Psi_{H_t}(\gamma(t), x)}{\Psi_{D_t}(\gamma(t), \infty)}, \qquad t<T. \]
The partition functions on the right-hand side are not well defined but the ratio is well
defined using the scaling rule,
\[\dfrac{\Psi_{D_t}(\gamma(t), x)}{\Psi_{D_t}(\gamma(t), \infty)} = \dfrac{\left |g'( \gamma_t)\right |^b \,g'(x)^b\,\Psi_{D_t}(U_t, g_t(x))}{\left |g'(\gamma_t)\right |^b \,g'(\infty)\, \Psi_{\Half}(U_t, \infty)} = g'_t(x)^b X_t^{1-3a}.\] While this is formal, this shows that we can define 
\[     M_t = \left(\frac{X_t}{X_0}\right)^{1-3a} \, g_t'(x)^b, \;\;\;\; t < T  \]
and  one can use It\^o's formula and the Loewner equation to see that  
$M_t$ is a local martingale satisfying $M_0 = 1$ and 
\begin{equation} \label{jul2.2}
       dM_t = \frac{1-3a}{X_t} \, M_t \, d B_t, \;\;\;\; 0 \leq t < T.
       \end{equation}
Let $\Prob^*$ be the measure obtained by tilting by $M_t$.  More precisely,
if $\tau < T$ is a stopping time such that $M_{t \wedge \tau}$ is a martingale
and $V$ is an event measurable with respect to $\F_{t \wedge \tau}$, then
$\Prob^*(V) = \E[M_{t \wedge \tau} \, 1_V]$. The Girsanov theorem states that
\[             dB_t = \frac{1-3a}{X_t} \, dt + d W_t , \;\;\;\; 0 \leq t < T , \]
where $W_t$ is a  standard Brownian motion with respect to $\Prob^*$
and hence
\[             dX_t =      \frac{1-2a}{X_t} \, dt + d W_t, \;\;\;\;
  0 \leq t < T.\]
 The following is well known.

 \begin{proposition}  Suppose $0 < x$ and $g_t$ is the solution to
the Loewner equation \eqref{loewner} where $U_t = g_t(x) - X_t$  and $X_t$ satisfies
\begin{equation}  \label{jul1.1}
     dX_t = \frac{1-2a}{X_t} \, dt + dW_t,\;\;\;X_0 = x, \;\;\;\; 0 \leq t < T ,
     \end{equation}
where  $W_t$ is a standard Brownian motion  and
$T = \inf\{t:X_t = 0\}$.  Then $\gamma(t), 0 \leq t \leq T$ has the
distribution of $SLE_\kappa$ from $0$ to $x$ parametrized by half
plane capacity from infinity stopped at the time that $\gamma_T$
disconnects $x$ from infinity.  In particular, $\hcap[\gamma_T] = aT$.

\end{proposition}

Indeed to verify this, one needs only check that the conformal image of
$SLE$ from $0$ to $\infty$ by a conformal transformation $F:\Half \rightarrow
\Half$ with $F(0) = 0, F(\infty) =x$ gives the same distribution on 
the driving function as \eqref{jul1.1}.

Similarly, if $X_t$ satisfies \eqref{jul1.1} and
we define $\tilde U_t = g_t(0) - X_t$ with corresponding
curve $\tilde   \gamma$, then 
  $\tilde \gamma(t), 0 \leq t \leq T$ has the
distribution of $SLE_\kappa$ from $x$ to $0$ parametrized by half
plane capacity from infinity stopped at the time that $\gamma_T$
disconnects $0$ from infinity.

While we may use the same $X_t$ for $SLE$ in both directions, the distribution
of the driving functions $U_t,\tilde U_t$ are different.  Indeed, $U_0 = x,
\tilde U_0 = 0$.  For this reason, we cannot conclude the reversibility immediately from this
fact.  One thing that does follow   is that the distribution of the
stopping time $T$ is the same for $SLE$ from $0$ to $x$ as for $SLE$ from $x$ to $0$.  It is
the same as the time to reach the origin for the Bessel process
\eqref{jul1.1}. 
Since  $a>   1/4$  the process reaches
the origin in finite time.

There is a significant difference between $\kappa \in (0,4]$ and $\kappa \in (4,8)$.
Let us consider $SLE$ from $0$ to $x $ with $x>0$
 stopped at time $T$.   The following statements
are with probability one with respect to the tilted measure $\Prob^*$. 
\begin{itemize}
\item If $0 < \kappa \leq 4$, then $\gamma(t), 0 \leq t \leq T$
is a simple curve with $\gamma(0) =0, \gamma(T) = x, \gamma(0,T) \subset \Half$.
\item If $4 < \kappa < 8$, then $\gamma(T) \in (x,\infty)$.  Although the $SLE$ curve continues
after time $T$,  $D_\infty$, the unbounded connected component of $\Half \setminus \gamma$
is the same as the unbounded connected component of $\Half \setminus \gamma_T$.
\end{itemize}
In particular, for $\kappa \leq 4$, the domain $D_\infty$ determines the entire curve while
for $4 < \kappa < 8$, the domain $D_\infty$ gives only the ``curve as viewed from infinity'',
that is $\gamma \cap \overline D_\infty$.
We will prove reversibility for the domain $D_\infty$.  In this paper, we do
the $\kappa \leq 4$ case reproving Zhan's result.

\begin{theorem} \label{maint} If $\kappa \leq 4$, the distribution of $D_\infty$, is the same for $SLE$
from $x_1$ to $x_2$ and for $SLE$ from $x_2$ to $x_1$.  Equivalently, the distribution
of the conformal transformation $g_\infty$ is the same.
\end{theorem}

Our proof is in the same spirit as Zhan's proof.  One novel aspect is that we choose a 
realization of the Bessel process \eqref{jul1.1} in a two step process: we first choose
a value   $T = t_0$ and then given $T$ we run the Bessel process conditioned so that
$T = t_0$.

If $X_t$ satisfies \eqref{jul1.1} where $W_t$ is a $\Prob^*$-Brownian motion, then
the transition probability of the process killed at the origin is
\[  q_s(x,y) = \frac{y}{x^{4a+1}\, s^{2a+\frac 12  } }\, \exp\left\{-\frac{x^2+y^2}{2s}
 \right\} \, h\left(\frac{xy}{s} \right), \]
 where $h = h_a$ is an entire function with $h(0) > 0$. 
The density of $T$  in the measure $\Prob^*$  is a constant times  
\begin{equation}\label{newyear}
  \phi(x,t) := x^{4a-1} \, t^{-\frac 12 -2a} \, \exp\left\{-\frac{x^2}{2t}
  \right\}. 
  \end{equation}
    The Bessel process conditioned so that $T = t_0$ is this
 process tilted by the $\Prob^*$-martingale
\begin{equation} \label{jul2.3}
 N_t := \phi(X_t, t_0-t), \qquad 0 \leq t <  t_0 
 \end{equation}
  which satisfies
\begin{equation}  \label{aug1.1}
   dN_t = N_t \left[\frac{4a-1}{X_t} - \frac{X_t}{t_0-t}\right] \, d W_t . 
   \end{equation}
Formally one can write $\phi(X_t, t_0-t) = \E^*[\mathbbm{1}_{T=t_0}|X_t]$ which can be thought of as a Doob martingale in the measure $\Prob^*$. Otherwise, the unconvinced reader may engage in a brief It\^o calculus exercise to derive
\eqref{aug1.1}.  Since $M_t$ is a $\Prob$ local martingale and $N_t$
is a $\Prob^*$ local martingale, we can see that $\tilde M_t := M_t \, N_t$
is a $\Prob$ local martingale.  Again, one can check this again using
It\^o calculus.  If we let
 \[   \tilde M_t= \tilde M_{t,t_0}  =
M_t \, N_t = x^{1-3a}\,  X_t ^{3a-1} \, g_t'(x)^{(3a-1)/2}
 \, \phi(X_t,t_0-t), \;\;\; 0 \leq t < t_0 , \]
then using \eqref{jul2.2} and \eqref{jul2.3} we see that $\tilde M_t, 0 \leq t < t_0$ is a $\Prob$-martingale
 satisfying
\[   d \tilde M_t = \left[\frac{ a}{X_t} - \frac{X_t}{t_0-t}\right] \,
 \tilde  M_t \, dB_t, \qquad
   0 \leq t <t_0. \]
If we tilt in the Girsanov sense as above by $\tilde{M}_t$ giving  the new measure $\hat{\Prob} $ we have 
 \[   dB_t = \left[\frac{a}{X_t} - \frac{X_t}{t_0-t}\right] \, dt + d \tilde W_t ,\]
\[    dX_t = d[g_t(x) + B_t] = \left[\frac{2a}{X_t} - \frac{X_t}{t_0 - t}\right] \, dt
    + d \tilde W_t, \]
where $\tilde{W}_t$ is a $\tilde{\Prob}$-Brownian motion.

\begin{definition}
 Suppose $x_1,x_2$ are distinct real numbers, $0 < \kappa < 8$,
  and $0 < t_0 < \infty$.  Then $SLE_\kappa$ from $x_1$ to $x_2$ in $\Half$
  of time duration $t_0$ is defined to be the solution of \eqref{loewner}
  where the driving function $U_t = g_t(x_2) - X_t$, and $X_t$
  satisfies
\begin{equation}  \label{conbessel}  
    dX_t =  \left[\frac{2a}{X_t} - \frac{X_t}{t_0 - t}\right] \, dt
    + d W_t, \;\;\;\;  X_0 = x_2 - x_1, \end{equation}
    \[  U_t = g_t(x_2) - X_t = x_2 + \int_0^t \frac{a\, ds}{X_s}
       - X_t , \]
where $W_t$ is a standard Brownian motion.  
  \end{definition}
  
 If $q_t(x,y)$ denotes the transition probability for a Bessel process
satisfying \eqref{jul1.1}, killed upon reaching the origin, then the density for a process satisfying \eqref{conbessel} is
  \[   \psi_t(x,y;t_0) = q_t(x,y) \, \frac{ \phi(y,t_0 - t)} {\phi(x,t_0)}.\] We will need one very believable fact about this process.  The proof uses standard techniques but we delay the proof to Section \ref{finalsec}.  This estimate
 is not optimal but will be more than sufficient for our purposes.
\begin{proposition}  \label{finalprop}
For every $0 < \kappa < 8$, there
 exists $c < \infty, u > 0$ such that if
 $X_t$ satisfies \eqref{conbessel},   
 then   for all $r > 0$, 
\[   \Prob\left\{\max_{0 \leq t \leq t_0}  |U_t-x_1|  \geq 
\sqrt{t_0} \,(|x_2-x_1| + r^2)
    \right\} \leq  c \, e^{-ur}.\]
  \end{proposition}

We denote the corresponding probability
  measure on paths (modulo reparametrization)   by $\mu^\#(x_1, x_2; t_0)$. Assuming $x_1 < x_2$, we have the following:
  \begin{itemize}
 \item  $\gamma(0) = x_1$, $T = t_0$;
 \item  $\hcap[\gamma_t] = at, \; 0 \leq t \leq T$;
 \item  If $\kappa  \leq  4$, then $\gamma_T$ is a simple curve with
 $\gamma(0,t_0) \subset \Half$ and $\gamma(t_0) = x_2$. Moreover,
 $\p D_\infty \cap \Half = \gamma(0,t_0) $;
 \item  If $4 < \kappa < 8$, then 
 \[     \gamma(T) = x_+:= \max\{y \in \R: y \in \gamma_{t_0}\} > x_2,\]
 \[       x_-:=  \min\{y \in \R: y \in \gamma_{t_0} \} < x_1 . \] Indeed,  $\p D_\infty  \cap \Half$ is a curve
 connecting $x_-$ to $x_+$. 
 \end{itemize}  
 
To prove Theorem \ref{maint} it suffices to prove the following.

 \begin{theorem}  If $\kappa \leq 4$ then for every $t_0 >0$ and $x_1 < x_2$,
the measure $\mu^\#(x_1,x_2;t_0)$ is the same as $\mu^\#(x_2,x_1;t_0)$
if considered as probability measures on the conformal transformation $g = g_{t_0}$.
\end{theorem} 

By scaling and translation invariance it suffices to prove this with $x_1 = 0,
x_2 = x > 0$ and $t_0 = 1$.

Fix $x>0$ and consider the measure $\mu_r =\mu_{r,x}$, $0 \leq r \leq 1$ obtained as follows: 
\begin{itemize}
    \item Grow the curve $\gamma$ under the measure $\mu^\#(0,x;1)$ until time $r$ giving curve $\gamma_r$ and corresponding map $g_r$.  Let $z_1 = g_r(\gamma(r)),
    w_1 = g_r( 1) $.
    \item Given $\gamma_r$, let $\tilde{\gamma}$  be $SLE$ from $x$ to $\gamma(r)$ in $\Half \setminus \gamma_r$ conditioned so that $\text{hcap}[\gamma_r \cup \tilde{\gamma}] = a$.  Equivalently, let $\eta$ be chosen from $\mu^\#(w_1,z_1;1-r)$ and let $\tilde \gamma = g_r^{-1} \circ \eta$. Let $h = g_\eta$ and
    $g = h \circ g_r$. 
\end{itemize}
Note that  $\mu_0 =\mu^\#(0,x;1),  \mu_1 =\mu^\#(x,0;1) $. 
We will prove the following stronger result, 
\begin{proposition}\label{aug2.prop2}
 If $\kappa \leq 4$ and $x > 0$, then for all
$0 \leq r \leq 1$, $\mu_r = \mu_0$.
\end{proposition}

 Since $\hcap[\gamma_1] = a$, we know that
$\gamma_1 \subset \{z: \Im(z)^2 \leq  {2a}\}$ and hence with probability one
for each $\mu_r$, $D_\infty \supset  \{z: \Im(z)^2>  {2a}\}$.
Let 
$\mathcal{I} = \{ z:  |z - ( \sqrt{8a} + 1)i| \leq 1\}$. 
Let $S$ denote the set of continuous functions from ${\mathcal I}$
to $\C$ endowed with the supremum norm $\|\cdot \|$.  
We also write $\rho$ for the corresponding Prokhorov metric on
probability measures on $S$.  Since the conformal map $g$ is determined
by its values on ${\mathcal I}$, it suffices to prove that for every
$\epsilon > 0$ and $0 \leq r < s \leq 1$,
 $\rho(\mu_r,\mu_s) < \epsilon$.   We will show the following.

\begin{proposition}  \label{aug2.prop1}
For every $K < \infty$, there exists $c,\delta$ such that  if $0 < x \leq K$
and   $0 \leq r \leq s \leq 1$, we can couple $(g,\tilde g)$
on the same probability space such that $g$ has distribution $\mu_r$, $\tilde g$ has distribution
$\mu_{s}$ and
\[           \Prob\{ \|g - \tilde g\| \geq  c \, (s-r)^{1+\delta } \} \leq c \, (s-r)^{\delta}, \]
\[          \Prob\{\|g - \tilde g\| \geq c \, (s-r) \} \leq c \, (s-r)^{1+\delta}.\]
\end{proposition}
We state it this way in preparation for later work in the $4 < \kappa < 8$ case.  For
$0 < \kappa \leq 4$, we do significantly better by giving a coupling that satisfies
$\|g - \tilde g\| \leq c \, (s-r)$ for all $(g,\tilde g)$ and such that
$\Prob\{\|g -\tilde g\| \geq (s-r)^{5/4}\}$ decays faster than every power of $s-r$.

Note that Proposition \ref{aug2.prop1}  implies that there exist $c,\delta$
\[  \rho(\mu_r, \mu_s) \leq c \, (s-r)^{1+ \delta}. \]
This shows that $\mu_r$ is H\"older continuous of order $1+\delta$ in $r$ and
a standard argument shows that this means that $\mu_r$ is a constant function
of $r$ and hence Proposition \ref{aug2.prop2} holds. 
%
%
%
%
%

\section{Local commutation relation}
\label{comm}

In this section we will state the basic ``commutation'' relation that we will
use.  In order to state the relation precisely we will set up some
notation.  Although we only use it for $\kappa \leq 4$ in this paper, we will
also give a result that holds for all $4 < \kappa < 8$.
 We fix $x_1 \neq x_2$ and $t_0 > 0$. Suppose $\gamma:[0,t_0]
\rightarrow \Half$
is a non-crossing curve parametrized by capacity from $x_1$ to $x_2$
in $\Half$. Let $\hat \gamma^R$ denote the reversed curve
from $x_2$ to $x_1$ defined by $\hat \gamma^R(t) = \gamma(t_0-t), \; 0 \leq t \leq t_0$.  Although  $\hat \gamma^R$ is not parametrized
by capacity, we can reparametrize it $\gamma^R(t) = \hat \gamma^R
(\sigma(t))$ so that for each $t$, $\hcap[\gamma^R_t] = at$. The total
time duration of $\gamma^R$ is the same as that of $\gamma$, $  t_0$.

If $0 < s_1 < s_2 < t_0$, we can write
\[          \gamma = \gamma_{s_1} \oplus \gamma[s_1, s_2]
 \oplus \gamma[ s_2,t_0], \]
 Let us write $\gamma^1$ for $\gamma_{s_1}$ and $\gamma^2$ for the reversal
 of $\gamma[s_2,t_0 ]$, so that we have
\begin{equation}  \label{mar3.1}
          \gamma = \gamma^1 \oplus \eta \oplus (\gamma^2)^R.
          \end{equation}
 Let us view this at the moment as a decomposition modulo reparametrization
 but still remember that $\hcap[\gamma] = at_0$ and we assume that
 \[ \hcap[\gamma^1 \cup (\gamma^2)^R] = \hcap[\gamma^1 \cup \gamma^2] <
     at_0 .\]
We will also assume that
\[           \gamma^1 \cap \gamma^2 = \eset . \]
If $\kappa \leq 4$, this will happen with probability one since $SLE_\kappa$
is supported on simple curves, but for $\kappa > 4$ this is a nontrivial
constraint.

Suppose  $r_1 + r_2 \leq t_0$, $V_1,V_2$ fixed subsets of $\C$,  and $\tau_1,\tau_2$ are stopping times for $\gamma^1,\gamma^2$  of
the form
\[    \tau_j  = \min\{s: \hcap[\gamma_s^j] = a\, r_j \mbox{ or } \gamma^j_s \not \in V_j\}. \]
We view probability measures on curves from $x_1$ to $x_2$
of half-plane capacity $at_0$ as probability measures on ordered pairs
\[   \bgamma =  (\gamma^1,\gamma^2):= (\gamma^1_{\tau_1},\gamma^2_{\tau_2}) . \]
 Here $\gamma^1,\gamma^2$ are parametrized by capacity, that is,
 $\hcap[\gamma^j_s] = as$.  Note that if $\gamma^1,\gamma^2$ are nontrivial, then
 \[    \hcap[\gamma^1 \cup  \gamma^2 ]
   < \hcap[\gamma^1] + \hcap[\gamma^2 ]
   \leq a(r_1+r_2) = t_0, \]
  and hence the $\eta$ in \eqref{mar3.1} is nontrivial. We will also
  assume that the stopping time is such that with probability one,
  $\gamma^1 \cap \gamma^2 = \eset$.  If $\kappa \leq 4$, t
  since $\eta$ is not trivial.  For $\kappa > 4$, we will guarantee it by
  choosing stopping times such that $\gamma^1 \subset V_1, \gamma^2 \subset
  V_2$ for some deterministic $V_1,V_2$ with $V_1 \cap V_2 = \eset$. 
  \begin{figure}
    \centering
    \includegraphics[scale=0.5]{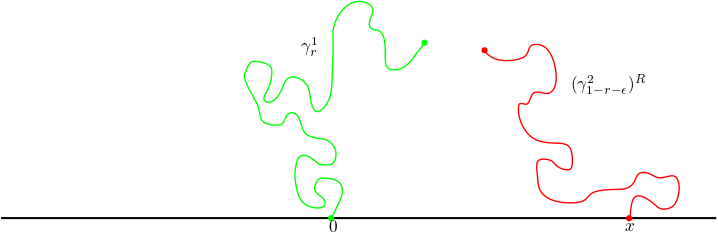}
    \caption{We grow $SLE$ from $0$ to $x$ until we reach hcap $ar$ (note that, by our choice of parametrization, this corresponds to time $r$). We then start $SLE$ from $x$ to $\gamma^1(r)$ stopped before its hcap reaches $a(1-r)$. The difference in the construction of the measures $\Prob_1^*$ and $\Prob_2^2$ comes in the middle piece, which we may construct in two ways.} 
    \label{fig1}
\end{figure}
  We now let ${\mathbb \Prob}_j^*$ be the probability measure on $\bgamma$
  given by
  \begin{itemize}
  \item Choose $\gamma^j$ from $SLE_\kappa$ from $x_j$ to $x_{3-j}$, conditioned
  to have total capacity $at_0$, stopped at time $\tau_j$.  Let $z_j=
  \gamma^j(\tau_j)$.
  \item Given $\gamma^j$, choose $\gamma^{3-j}$ from $SLE_\kappa$ from $x_{3-j}$
  to $z_{j}$ in $\Half \setminus \gamma^{j}$, conditioned to that the total
  capacity of the union of the curve and $\gamma^j$ is $at_0$, stopped
  at time $\tau_{3-j}$.
  \end{itemize}
  The commutation result is that ${\mathbb P}_1^* = {\mathbb P}_2^*$.
  We sketch the proof by giving the Radon-Nikodym derivative of each
  of the measures with respect to $\Prob$, the measure obtained from independent
  $SLE_\kappa$ paths.   To state this we give some notation. 
  Let  $D^j =
 \Half \setminus \gamma^j ,   D = \Half \setminus
 \bgamma $.  Let $g^j,g$ be the corresponding conformal maps; let
 $z_j = \gamma^j(\tau_j), U^j = g^j(z_j)$ and define
 $h_2,h_1$ by $h_2 \circ g^1 = g = h_1 \circ g_2$.

 \begin{figure}[ht]
     \centering
     \includegraphics[scale=0.4]{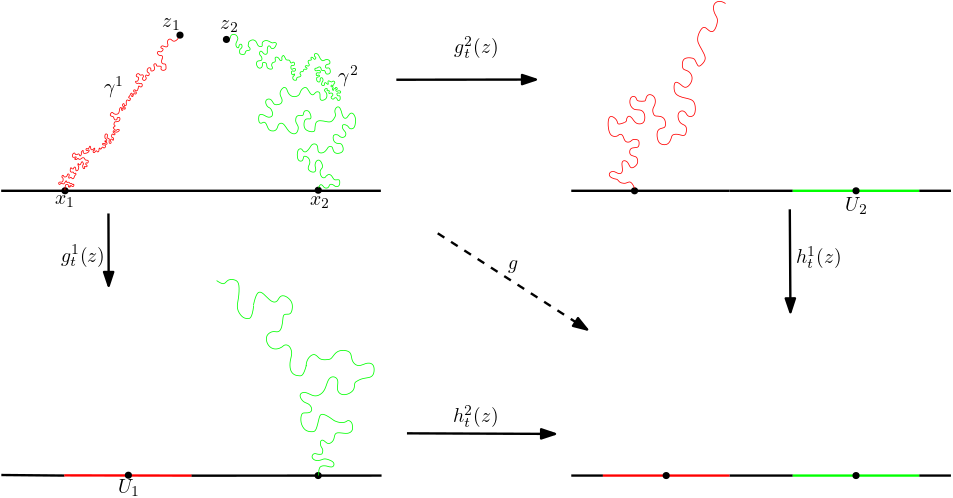}
     \caption{The maps $g^1, g^2, h^1$ and $h^2$ exhibit a commutative relation.}
     \label{fig2}
 \end{figure}
 
 \begin{proposition} \label{commute}
  The  Radon-Nikodym derivative of $\Prob_j^*$ with respect
 to $\Prob$, the measure obtained from independent $SLE$ paths from $0$
 to infinity stopped at times $\tau_1,\tau_2$,  is given by
   \[  \frac{d  \Prob_j^*}{d\Prob} (\bgamma) = 
     {h_1'(U^2) ^b \, h_2'(U^1)^b  } \,  \exp\left\{\frac \cent 2 \, m_\Half(\gamma^1,\gamma^2)
\right\}
\, \frac{|g(z_2) - g(z_1)|^{2b}}
       {|x_2-x_1|^{2b}}\, \frac{\phi( |U^2 - U^1 |,t_0 - \tau_1 + \tau_2)}{\phi(|x_2-x_1|,1)}. \]
Here $b = (6-\kappa)/2\kappa$ is the boundary scaling exponent, $\cent = (6-\kappa)(3\kappa - 8)/2\kappa$ is the central charge,
and  $m_\Half(\gamma^1,\gamma^2)$ denotes the        
       Brownian loop measure of loops in $\Half$ that intersect both $\gamma^1$ and $\gamma^2$
   and $\phi$ is as in \eqref{newyear}.   In particular, $\Prob_1^* =
   \Prob_2^*$.
   \end{proposition}

  \begin{proof} Without loss of generality, we assume $t_0 = 1$. We will
  prove the result for $j=1$.

\begin{itemize}

\item  We start by choosing $  \gamma^1 $ using $SLE_\kappa$
from $x_1$ to $x_2$ stopped at time $\tau_1$.  Here we are not conditioning on the
total time duration of the path.  The Radon-Nikodym derivative of this with
respect to $SLE$ from $0$ to infinity, restricted to the event that the total
time duration is greater than $\tau_1$ is 
\[       g_1'(x_2)^{b}
  \, \frac{|g_{1}(x_2) - U_t^1|^{2b}}{|x_2-x_1|^{2b}}.\]
  Let $\eta_2 = g_1\circ \gamma^2 $.

\item  Given $  \gamma^1 $, we will choose
$ \gamma^2 $ using $SLE$ from $x_2$ to $z_1$
in the domain $D_1$.  We will do this in two
steps.

\item  We first choose  $  \gamma^2$  using $SLE$ from $x_2$ to  infinity
in  $D_1$.  Using the basic martingale of the restriction property
this gives Radon-Nikodym derivative
\[              \exp\left\{\frac \cent 2 \, m_D(\gamma^1,\gamma^2)
\right\}  \, \frac{h_1'(U^2) ^b }{g_1'(x_2)^{b}}.\]
Note that $\eta_2 := g_1 \circ \gamma_2$ is an $SLE$ from $g_1(x_2)$
to infinity.

\item We now tilt again so that  $\eta_2 := g_1 \circ \gamma_2$ is an $SLE$ from $g_1(x_2)$
to  $U^1$.  This gives a Radon-Nikodym derivative
\[    {h_2'(U^1)}^b \, \frac{|h_2(\eta_2(\tau_2)) - h_2(U_1)|^{2b}}
       {|g_1(x_2) - U^1|^{2b}}
        =  {h_2'(U^1)}^b \, \frac{|g(z_2) - g(z_1)|^{2b}}
       {|g_1(x_2) - U^1|^{2b}}. \]
       
\item Multiplying the last two gives
\[  \exp\left\{\frac \cent 2 \, m_D(\gamma^1,\gamma^2)
\right\}  \, \frac{h_1'(U^2) ^b \, h_2'(U^1)^b  }{g_1'(x_2)^{b}}
\, \frac{|g(z_2) - g(z_1)|^{2b}}
       {|g_1(x_2) - U^1|^{2b}}. \]
       
 \item  We thus have that the Radon-Nikodym derivative restricted to the event that
 the total time duration is greater than $\tau_1 + \tau_2$ is given by:
 \[      {h_1'(U^2) ^b \, h_2'(U^1)^b  } \,  \exp\left\{\frac \cent 2 \, m_D(\gamma^1,\gamma^2)
\right\}
\, \frac{|g(z_2) - g(z_1)|^{2b}}
       {|x_2-x_1|^{2b}}. \]
       
  \item If we now condition so that the total time duration is one we get
  
   \[      {h_1'(U^2) ^b \, h_2'(U^1)^b  } \,  \exp\left\{\frac \cent 2 \, m_D(\gamma^1,\gamma^2)
\right\}
\, \frac{|g(z_2) - g(z_1)|^{2b}}
       {|x_2-x_1|^{2b}}\;  \frac{\phi( |U^2 - U^1 |,1 - (\tau_1+\tau_2))}{\phi(|x_2-x_1|,1)}. \]
       \end{itemize}

\end{proof}

\begin{figure}
    \centering
    \includegraphics[scale=0.3]{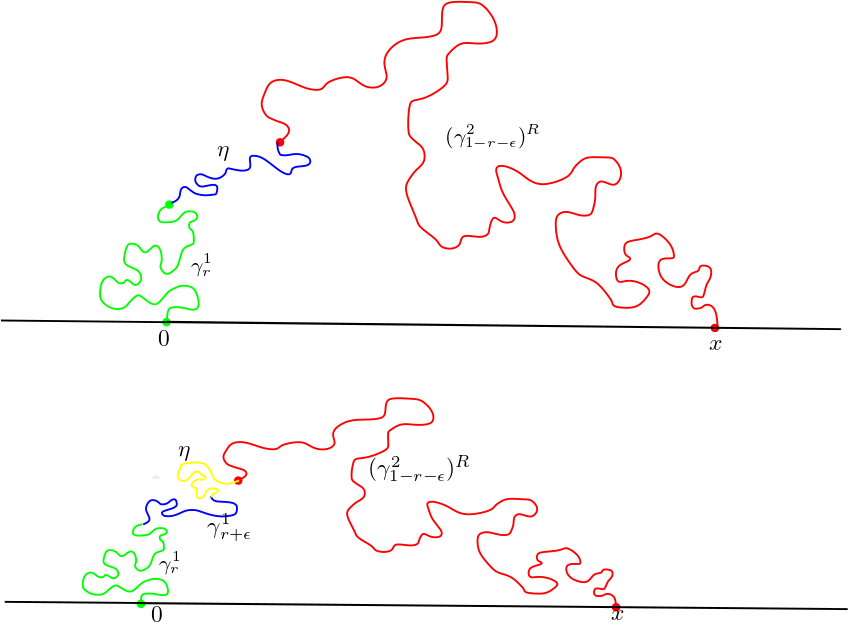}
    \caption{The difference in the construction comes in the curves $\eta $ and $\tilde{\eta}$. Given this, we may sample from $\mu_r$ and $\mu_{r + \epsilon}$ respectively, allowing us to conclude using basic facts about the Loewner equation.}
    \label{fig3}
\end{figure}

\section{Proof of main Theorem}
\label{main}

We will use some basic facts about the Loewner equation.

\begin{proposition} \cite[Proposition 3.46]{lawler-book} There exists $c < \infty$ such that
if   $D  = \Half \setminus K$ is a simply connected domain with
$r = \sup\{|z|: z \in K\}  $ and $h  = \hcap(K)  $,  then the  corresponding conformal
map $g: D \to \Half $ satisfies for $|z| \geq 2 r$, 
\[     \left|g_D(z) - z - \frac{h}{z}\right| \leq \frac{c\,rh}{|z|^2}.
  \]
 In particular, if $K,\tilde K$ are two such hulls with $h = \tilde h$, then 
  for   $|z| \geq 2 (r   \wedge \tilde r),$
  \[   \left|g (z) -\tilde  g (z)\right| \leq \frac{c \, (r + \tilde r) \,h}
   {|z|^2} . \]
\end{proposition}

\begin{proposition} \cite[Proposition 4.13]{lawler-book} There exists $c < \infty$, such that
if  $U_t$ is a driving function with $U_0 = 0$
and $\gamma_t$ is the corresponding curve, then 
\[   \diam[\gamma_t] \leq c \, \left[\sqrt t + \max_{0 \leq s \leq t} |U_s|\right].\]
\end{proposition}

We also need some easy estimates about our Bessel process conditioned to reach
the origin at a given time.

 \begin{lemma}  If $K < \infty$, there exists $\epsilon_0 > 0$ such that
 if  $X_t$ satisfies \eqref{conbessel} with $t_0 = 1$ and $|x_2 - x_1| \leq K $,
then as $\epsilon \rightarrow 0$, 
\[                 \Prob\left\{|X_{1-\epsilon}| \geq   \sqrt{\epsilon}  \, \log(1/\epsilon) 
    \right\}  \]
    decays faster than every power of $\epsilon$.
\end{lemma}

\begin{proof} By a coupling argument, the probability on the left restricted
to $|x_2-x_1| \leq K$ is maximized when $x_2 - x_1 = K$.  In this case, we
can look at the transition probability.
\end{proof}

We write $\epsilon =s-r$.  We decompose a simple path $\gamma$ from $0$ to $x$ with $\hcap[\gamma]
= a$ as 
\[           \gamma = \gamma^1 \oplus \eta \oplus  (\eta')^R \oplus (\gamma^2)^R ,\]
where the decomposition is defined by 
\[  \hcap[\gamma^1] = ra,  \;\;\;\;\; \hcap[\gamma^1 \cup \eta] = sa, \;\;\;\;
\;\;\;\;  \hcap[ \gamma^1 \cup   \gamma^2] = (t_0-\epsilon)a.\]
Using the definition and the conformal Markov property, we can see
that when we  sampling from $\mu_s$ we choose the paths in order $\gamma^1,\eta,\gamma^2,\eta'$.  When we
sample from $\mu_r$ we use the order $\gamma^1,\gamma^2,\eta',\eta^R$.  In
each case the distribution is $SLE$ to the endpoint of the other curve in
the domain slit by the curves at that point, conditioned to have  the
appropriate total
half-plane capacity and stopped as specified above.

We now use  Proposition \ref{commute} to say
that another way  to sample from $\mu_s$ is to   choose the
paths in order $\gamma^1,\gamma^2, \eta, \eta'$.   Hence we can write the sampling as
follows. Steps 1 and 2 are the same for both sampling methods.
 Step 3a is used for $\mu_s$ and Step 3b is used for $\mu_r$. 

\begin{itemize}
\item {\bf Step 1:}  Choose $\gamma^1$ from $SLE$ from $0$ to $x_0$ conditioned to have
total half-plane capacity $a$ stopped at time $r$, that is, stopped when $\hcap[\gamma^1]
 = ar$.  Let $z_1 = \gamma(r)$, let
  $\hat g : \Half \setminus \gamma^1 \rightarrow \Half$ be the corresponding transformation, and let $y_1 = \hat g(z_1), x_1 = \hat g(x_0)$.
 
\item {\bf Step 2:} Choose $\eta$ from $SLE$ from $x_1$ to $y_1$ conditioned to have
total half-plane capacity $a(1-r)$ stopped at time $1-s$, that is, stopped when
$\hcap[\eta] = a(1-s)$. Let $h: \Half \setminus \eta \rightarrow \Half$ be the corresponding
transformation, and let $y_2 = h(y_1), x_2 = h(\eta(1-s))$.  Let $\gamma^2 =
\hat g^{-1} \circ \eta$ and $w_1 = g^{-1}(\eta(1-s))$.
  Let $\hat h = h \circ g$ and note that $\hat h: \Half \setminus
(\gamma^1 \cup \gamma^2) \rightarrow \Half$ is the corresponding conformal transformation
which satisfies $\hat h(z_1) = y_2, \hat h(w_1) = x_2$.

\item {\bf Step 3a:} Choose $\omega^1$ from $SLE$ from $y_2$ to $x_2$ conditioned to have
total half-plane capacity $a\epsilon$ stopped at the first time that
\[                     \hcap [h^{-1} \circ \omega^1] = a \epsilon . \]
This is the same as the first time that
\[                    \hcap[\gamma^1 \cup \hat h^{-1} \circ \omega^1] = as.\]
 Let this time be $u $ and let $\phi: \Half \setminus \omega^1 \rightarrow \Half$
 be the corresponding transformation with $y_3=\phi(\omega^1_u), x_3 = \phi(x_2)$.  Let
 $\tilde \omega^2$ be chosen from $SLE$ from $x_3$ to $y_3$ conditioned to have
 half-plane capacity $a(\epsilon - u)$ giving conformal map $\hat \phi$ and let
 $\omega^2 = \hat \phi^{-1} \circ \tilde \omega^2$ and
 \[                   \omega = \omega^1 \oplus [\omega^2]^R . \]
 Let $\psi: \Half \setminus \omega \rightarrow \Half$ be the corresponding conformal
 transformation.
 \begin{figure}
     \centering
     \includegraphics[scale=0.5]{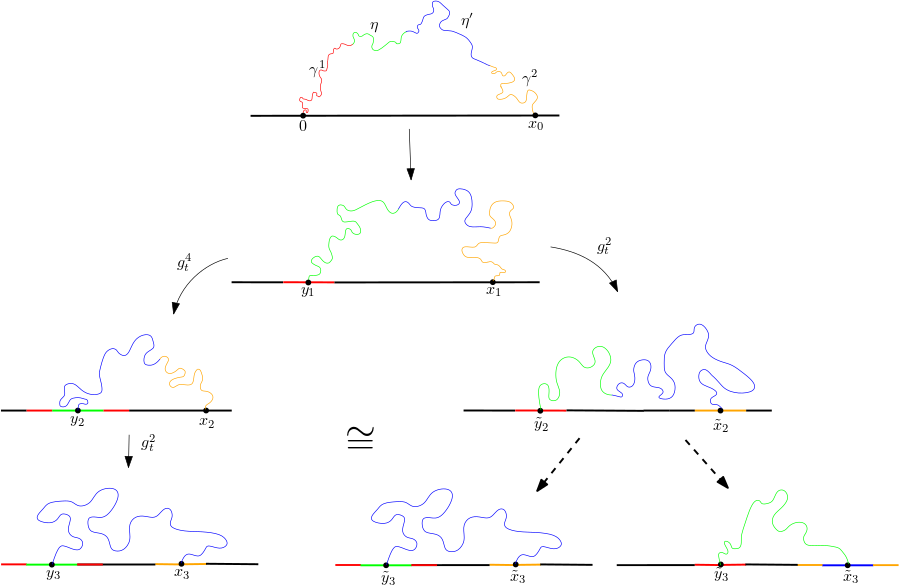}
     \caption{A schematic showing the full picture, though not drawn to scale (in particular, the yellow and blue segments ought not to have comparable lengths). The dotted arrows on the right correspond to the commutation relation, and together with some Loewner estimates, we may conclude that the laws of the measures obtained, regardless of the path one chooses in the schematic, are the same.}
     \label{fig:my_label}
 \end{figure}
 \item {\bf Step 3b} Choose $\omega^*$ from $SLE$ from $x_2$ to $y_2$ conditioned to
 have total half-plane capacity $a\epsilon$ and set
 \[                    \hat \omega = [\omega^*]^R.\]
 Let $\hat \psi: \Half \setminus \hat \omega \rightarrow \Half$ be the corresponding
 conformal transformation.

 \end{itemize}

In our coupling we use the complete coupling for steps 1 and 2. Hence we write
\[    g = \psi \circ h, \;\;\;\;\; \tilde g = \tilde \psi \circ h , \]
where $h$ is the same in both cases.  If $z \in {\mathcal I}$, then 
$\Im(h(z)) \geq \sqrt{4a}$.    Except for an event
of probability that decays faster than every power of $\epsilon$, we have
$x_2 - y_2 \leq \epsilon^{1/2} \, \log (1/\epsilon)$.  Using this, we see that in step 3a
and in step 3b we get a curve with the same initial and terminal points, of half
plane capacity $a \epsilon$ and such that, except for an event of probability that decays faster than
every power of $\epsilon$, has diameter bounded by $\epsilon^{1/2} \, \log^2 \epsilon $.
Let $\psi, \tilde \psi$ be the conformal transformations.  Then if $\Im(z) \geq \sqrt a$
we have
\[                |\psi(z) - \tilde \psi(z)| \leq c \, \epsilon , \]
and, except for an event of probability that decays faster than
every power of $\epsilon$,
\[                    |\psi(z) - \tilde \psi(z)| \leq   \epsilon^{5/4}.\]
Therefore, in this coupling, with probability one $\|g - \tilde g \| \leq c \, \epsilon$
and
\[               \Prob\{\|g - \tilde g\| \geq \epsilon^{5/4}\}    \leq c \, \epsilon^3.\]

\section{Proof of Lemma \ref{finalprop}}  \label{finalsec}

  We fix
$a > 1/4$ and allow constants to depend on $a$.  
 We assume
that $X_t$ satisfies \eqref{conbessel}.  For ease we will
assume $x >0$ but the proof with $x < 0$ is essentially
the same.

The proof  follows from the easy estimate
\[   -X_t \leq U_t \leq  \int_0^t \frac{a}{X_s} \, ds \]
 and the following
two lemmas that handle the two sides of the inequality.  For the lower
bound, we get a somewhat sharper estimate. 

\begin{lemma}
There exists $c < \infty$ such that if
 $X_t$ satisfies \eqref{conbessel} with   $X_0 = x_0 \, \sqrt{t_0} >0,$
 then  for all $r > 0$, 
\[   \Prob\left\{\max_{0 \leq t \leq t_0} (X_t/\sqrt{t_0}) \geq x_0 + r
    \right\} \leq  c \, \exp\left\{-\frac{r^2}{4} \right\}.\]
\end{lemma}

\begin{proof} We may assume
that $r^2 \geq 1+4a$ and by scaling we may assume $t_0 = 1$.
  Let $y = 
  x_0 +r$  
   and let  $\sigma =  \inf\{t: X_t=  y\}.$   The equation \eqref{conbessel} can be obtained
by starting with $X_t$ satisfying \eqref{jul1.1} where $W_t$
is a $\Prob^*$ Brownian motion  and then tilting
by the martingale $N_t$ as in \eqref{aug1.1} to get the measure
$\Prob$.  
Hence, 
\[ \Prob \{\sigma  < 1 \}
  \leq M_0^{-1} \, \E^*\left[M_{\sigma_\epsilon}; \sigma_\epsilon < \infty\right]   .\]
 Note that
 \[   M_0 = x_0^{4a-1}  \, \exp\left\{-\frac{x_0^2}2\right\} ,\]
 and if $\sigma <1$, 
  \[  M_{\sigma} \leq   \max_{0 \leq t \leq 1}
                  y^{4a-1} \, (1-t)^{-\frac 12 - 2a} \, \exp
                   \left\{- \frac{y^2}{2(1-t)}\right\}
                    = y^{4a-1} \, e^{-y^2/2} .\]        
The equality uses $r^2 \geq 1+4a$.  Therefore,
\[ 
  \frac{M_\sigma}{M_0} \leq \left[1 + \frac {\log(1/\epsilon)}{x_0}
  \right]^{4a-1} \, \exp \left\{- x_0 \, r -
     \frac{r^2 }{2} \right\}  
      \leq c\, \exp\left\{-\frac{r^2 }{4} \right\}. 
\]
\end{proof}

\begin{lemma}
If $a > 1/4$, there exist $u > 0$ and $c < \infty$ such that for any
$x > 0$ and $t_0 >0$ if $X_t$ satisfies \eqref{conbessel},
then for all $r > 0$,
\[  
  \Prob^x\left\{
    \int_0^{t_0} \frac{ds}{X_s} \,  dy
    \geq  r\sqrt{t_0}  \right\}  \leq c\,  e^{-ur}. \]
 
\end{lemma}

\begin{proof}

Let 
\[   I_n = \int_{0}^{ 1} \frac{ds}{X_s} \, 1\{2^{-n} \leq X_s
   < 2^{-n+1} \} \, ds.\]
   Our first goal is to show that there exists $c_* < \infty$ such
   that for all $x, t_0,n$,
\begin{equation}
\label{aug23.1}
   \E^x[I_n] \leq c_* \, 2^{-n} , \;\;\;\;\;
     \Prob^x\{I_n \geq c_* \, 2^{-n+1} \} \leq \frac 12 . 
   \end{equation}
   The second follows from the first by the Markov property; by scaling, 
  It suffices to show the first inequality  for $n=0$.
 By the strong Markov property.
we may assume that $1  \leq x \leq 2$; otherwise, we first run the process until
it reaches $[1,2]$.  Also, note that
\[      \E^x[I_0] \leq     \int_{1}^{2} \int_0^{t_0} \, \phi_t(x,y;t_0) \, dt \, dy .\] Using the immediate estimate
\[     \int_{1}^2\left[ \int_{t_0 - 1}^{t_0} \, \phi_t(x,y;t_0) \, dt 
 +   \int_{0}^{1} \, \phi_t(x,y;t_0) \, dt \right]dy \leq 2,\]
we see that it suffices to show that there exists $c$ such that for all
$1 \leq x,y \leq 2$ and $t_0 \geq 1$,
\[                \int_1^{t_0-1}  \phi_t(x,y;t_0) \, dt \leq c . \]
This can be done in a straightforward way by looking at the transition probability.
Indeed, if $1 \leq s \leq t_0 -1$ and $1 \leq x,y \leq 2$,
\[ \phi_t(x,y;t_0) \leq c \, \left[\frac{t_0}{t_0-t}\right]^{2a+\frac 12}\, \frac{1}{
t^{2a + \frac 12}}
  .\]
 For $t < t_0/2$ we estimate this by $c \, t^{-(2a + \frac 12)}$ and for
 $t \geq t_0/2$, we estimate this by $c\, (t_0 - t)^{-(2a + \frac 12)}.$
 Provided that $a > 1/4$ we see that this integral is uniformly bounded
 in $t_0$.  This gives \eqref{aug23.1}.

By scaling  it suffices to prove our main result for $t_0 = 1$.  Note that
\[            
 \int_0^{1} \frac{ds}{X_s} \,  ds \leq 1 + \sum_{n=1}^\infty I_n , \]
 where
 \[   I_n = \int_{0}^{ 1} \frac{ds}{X_s} \, 1\{2^{-n} \leq X_s
   < 2^{-n+1} \} \, ds.\]
  By iterating \eqref{aug23.1} using the strong Markov property, we
  see that for all positive integers $k$, $\Prob\{I_n \geq 2k \, c_*
   \, 2^{-n}\} \leq 2^{-k}$ and hence for all $r >0$,
   \[             \Prob\{I_n \geq r \, 2^{-n } \} \leq c'\, e^{-ur}, \]
  where $u = (\log 2)/(2c_*), c' = e^u$. In particular,
 \[  \Prob\left\{\sum_{n=1}^\infty I_n \geq 2r\right\}
   \leq \sum_{n=1}^\infty \Prob\left\{I_n \geq  r (2/3) 
  ^n \right\}
      \leq c'\sum_{n=1}^\infty \exp\left\{-ur(4/3)^n
      \right\} \leq c \, e^{-2u(2r)/3}.\]

\end{proof}
\bibliographystyle{plain}
\bibliography{cibib}
\end{document}